\pdfoutput=1
\documentclass{amsart}

\usepackage{url,amssymb,amsmath,amsthm,amscd,paralist}
\usepackage{tkz-graph}
\usepackage{tikz-cd}
\usepackage{algpseudocode}
\usepackage{fancyhdr}
\usepackage{amssymb}
\usepackage{graphicx}
\usepackage[T1]{fontenc}
\usepackage[numbers]{natbib}
\usepackage{fullpage}
\tikzstyle{vertex}=[circle,black, fill=black, draw, inner sep=0pt, minimum size=6pt]
\newcommand{\vertex}{\node[vertex]}
\usepackage{tikz}
\usetikzlibrary{decorations.pathreplacing}
\usepackage{xcolor}
\definecolor{cof}{RGB}{219,144,71}
\definecolor{pur}{RGB}{186,146,162}
\definecolor{greeo}{RGB}{91,173,69}
\definecolor{greet}{RGB}{52,111,72}

\newcommand{\R}{\mathbb{R}}

\newtheorem{thm}{Theorem}

\newtheorem{prop}[thm]{Proposition}

\theoremstyle{definition}

\theoremstyle{remark}

\DeclareMathOperator{\vol}{vol}

\begin{document}

%\author{Beauttie Kuture \\ Oscar Leong \\ Christopher Loa \\ Mutiara Sondjaja \\ Francis Edward Su}
\author{Beauttie Kuture}
\address[Kuture]{
Department of Mathematics, Pomona College, Claremont, CA 91711, USA}
\email{bak02013@mymail.pomona.edu}
\author{Oscar Leong}
\address[Leong]{
Department of Mathematics, Swarthmore College, Swarthmore, PA 19081, USA}
\email{oleong1@swarthmore.edu}
\author{Christopher Loa}
\address[Loa]{
Department of Mathematics, University of Tennessee, Knoxville, TN 37916, USA}
\email{cloa@vols.utk.edu}
\author{Mutiara Sondjaja}
\address[Sondjaja]{
Department of Mathematics, New York University, New York, NY 10012, USA}
\email{sondjaja@cims.nyu.edu}
\author{Francis Edward Su}
\address[Su]{Department of Mathematics, Harvey Mudd College, Claremont, CA 91711, USA}
\email{su@math.hmc.edu}

\thanks{Su is the corresponding author. Email: su@math.hmc.edu}
\thanks{This work was 
%conducted at the Mathematical Sciences Research Institute and 
conducted during the 2015 Mathematical Sciences
Research Institute Undergraduate Program (MSRI-UP),
supported by grants
from the National Science Foundation (DMS 1156499)
and the National Security Agency (H98230-15-1-0039).
We also thank Duane Cooper for his encouragement.}
\title{Proving Tucker's Lemma With a Volume Argument}

\begin{abstract}
Sperner's lemma is a statement about labeled triangulations of a simplex.  McLennan and Tourky (2007) provided a novel proof of Sperner's Lemma by examining volumes of simplices in a triangulation under time-linear simplex-linear deformation. We adapt a similar argument to prove Tucker's Lemma on a triangulated cross-polytope $P$. The McLennan-Tourky technique does not directly apply because this deformation may distort the volume of $P$.  We remedy this by inscribing $P$ in its dual polytope, triangulating it, and considering how the volumes of deformed simplices behave.
\end{abstract}
\maketitle

\section{Introduction}

%\subsection{Background}
Sperner's Lemma is a combinatorial result that is equivalent to the Brouwer fixed point theorem and has many useful applications in mathematics and economics, including Nash's proof of his famous equilibrium theorem \cite{nash}. Consider an $d$-dimensional simplex $S$ with a \emph{triangulation}, i.e., a subdivision into smaller simplices that meet face-to-face or not at all. Sperner's Lemma states that if such a triangulation has a {\em Sperner labeling} (each vertex is labeled by one of the vertices of $S$ that span the minimal face that the vertex is on), then there exists an odd number of \emph{fully-labeled} simplices (ones whose vertices have distinct labels). 

Recently, McLennan and Tourky \cite{mclennan-tourky} gave a novel proof of Sperner's Lemma based on the following facts:

\begin{enumerate}
	\item A triangulation remains a triangulation under a small perturbation that keeps each vertex in its minimal face.
    \item The (signed) volume of a simplex spanned by $v_0, v_1, ..., v_d$ is $1/d!$ times the determinant of a matrix whose columns are $v_i-v_0$ for $i=1,..., d$.
    \item A polynomial function of $t$ is constant if it is constant on a nonempty open $t$-interval.
\end{enumerate}

Given a Sperner-labeled triangulation of a simplex, they construct a continuous deformation that moves each vertex of the triangulation linearly in time to the vertex of $S$ that it is named after.
By Fact (1), for small $t$ the sum of deformed volumes of all simplices is equal to the total volume of $S$, so the sum is constant for small $t$. By Fact (2), this volume sum is also a polynomial in $t$ because the volume of each deformed simplex is the determinant of a matrix whose entries are linear in $t$.  By Fact (3) the volume sum remains constant for all $t$, hence at $t=1$ it is also non-zero and represents the sum of volumes of all deformed simplices. Thus some deformed simplex has non-zero volume, which means it was originally a fully-labeled simplex (because non-fully-labeled simplices become degenerate under the deformation).
In fact, there must be one more fully-labeled simplex of positive volume than of negative volume, which implies that the number of fully-labeled simplices is odd. 

While there are many proofs of Sperner's lemma (e.g., see \cite{su}), the novelty of the McLennan-Tourky approach is that it uses a geometric volume argument rather than an inductive combinatorial argument to find the desired fully-labeled simplex. In this paper, we show that this approach can be used to prove another classical lemma from topological combinatorics.

Like Sperner's lemma, Tucker's Lemma \cite{lefschetz, tucker} is a combinatorial analog of a topological theorem, in this case, the Borsuk-Ulam theorem.  It also has beautiful connections to other well-known results in mathematics and economics, such as the ham sandwich theorem, necklace-splitting problems \cite{meunier}, and fair division problems \cite{simmons-su}.

Let $e_i$ denote the standard unit basis vector in $\R^d$ (with $1$ in the $i$-th coordinate and $0$ in every other coordinate) and define $e_{-i} = -e_i$.  Let $P$ denote the $d$-dimensional \emph{cross-polytope}, the convex hull of the $2n$ vectors $\{ e_i, e_{-i}\}_{i=1}^d$.  

\begin{thm}[Tucker's Lemma \cite{tucker}]
\label{mainthm}
Let $T_P$ be a triangulation of the $d$-dimensional cross-polytope $P$ whose restriction to the boundary $\partial P$ is antipodally symmetric: if $\sigma$ is a simplex in $\partial P$, then $-\sigma$ also is a simplex. 
Suppose that $T_P$ has a \emph{Tucker labeling}: each vertex $v$ of $T_P$ is assigned a label $\ell(v)$ in $\{\pm 1, \pm 2,... \pm d\}$ such that antipodal vertices in $\partial P$ have labels summing to zero. Then there is an \emph{complementary edge}: a pair of adjacent vertices in $T_P$ have labels summing to zero.
\end{thm}

\begin{figure}[ht]
\begin{tikzpicture}[scale=1]

	\vertex (v1) at (2,0) {};
    \node[label=$+1$] at (1.9,-0.1) {};
    \vertex (v3) at (2,-4) {};
    \node[label=$-1$] at (1.9,-4.7) {};

    \vertex (v2) at (4,-2) {};
    \node[label=$+1$] at (4.4,-2.5) {};
    \vertex (v4) at (0,-2) {};
    \node[label=$-1$] at (-0.4,-2.5) {};    
    
    \vertex (v5) at (3,-1) {};
        \node[label=$+2$] at (3.3,-1.2) {};
    \vertex (v6) at (1,-3) {};
        \node[label=$-2$] at (0.6,-3.7) {};
        
    \vertex (v7) at (1,-1) {};
        \node[label=$-2$] at (0.7,-1.1) {};
    \vertex (v8) at (3,-3) {};
        \node[label=$+2$] at (3.3,-3.7) {};

    \vertex (v11) at (0.5,-1.5) {};
        \node[label=$-2$] at (0.1,-1.7) {};
    \vertex (v12) at (3.5,-2.5) {};
        \node[label=$+2$] at (3.8,-3.2) {};
    
    \vertex (v13) at (1.7,-2.3) {};
    \node[label=$-2$] at (2,-3) {};
    \vertex (v14) at (2.9,-2) {};
    \node[label=$+2$] at (3.3,-2.2) {};
    \vertex (v15) at (1.9,-1.5) {};            
    \node[label=$+1$] at (2.1,-1.7) {};

	 \tikzset{EdgeStyle/.style={-}}
	\Edge(v1)(v2)
    \Edge(v2)(v3)
    \Edge(v3)(v4)
    \Edge(v4)(v1)
    
    \Edge(v1)(v14)
    \Edge(v2)(v14)
    \Edge(v5)(v14)
    \Edge(v7)(v13)
    \Edge(v13)(v15)
    
    \Edge(v6)(v13)
    \Edge(v4)(v13)
    \Edge(v3)(v13)
    \Edge(v3)(v14)
    \Edge(v11)(v13)
    \Edge(v7)(v15)
    \Edge(v15)(v14)
    \Edge(v14)(v12)
    \Edge(v12)(v14)
    \Edge(v8)(v14)
    \Edge(v14)(v13)
    \Edge(v15)(v1)
    
\end{tikzpicture}
\caption{A Tucker labeling of the cross-polytope $P$ must have a complementary edge.}
\label{figure:tucker}
\end{figure}
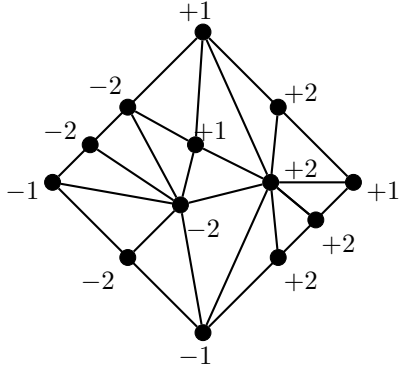

Although the cross-polytope $P$ in Tucker's lemma is convenient for our purposes, the result holds if $P$ is any centrally symmetric object topologically equivalent to a ball. Tucker \cite{tucker} established his lemma for a triangulated 2-dimensional ball, and Lefschetz \cite{lefschetz} gave a general proof for a $d$-dimensional ball. 
Baker \cite{baker} developed a cubical version.
Freund-Todd \cite{freund-todd} gave the first constructive proof of Tucker's lemma by using a path-following argument on a cross-polytope. Prescott-Su \cite{prescott-su} obtained a different constructive proof on a triangulated ball that proves also the more general Fan's lemma.

At first glance, it may seem that the McLennan-Tourky approach can be directly applied to prove Tucker's Lemma---just deform the Tucker-labelled triangulation by moving every vertex to the extreme point with the same label.  However, such a deformation, even for small time intervals, is not guaranteed to cover $P$, so the sum of the volumes of the deformed simplices will not necessarily remain constant.  This is the new challenge that didn't exist with Sperner's lemma.

We avoid this problem by placing the cross-polytope $P$ inside a larger polytope $C$ and applying the McLennan-Tourky technique to $C$. The deformation we construct will continue to cover $C$ for an open $t$-interval, so the total volume of the simplices remains constant and equal to the total volume of $C$ even after the deformed simplices no longer form a triangulation.  We observe what happens inside $P$.
A volume argument will show that if the boundary of $P$ has no complementary edges, then the interior of $P$ must be covered by deformed simplices of $P$ at time $t=1$.  The only way $P$ can be covered by simplices is if one of them had two labels that summed to zero.  This would yield the desired complementary edge.  The goal of this paper is to demonstrate how to make this intuition precise. 

\section{Volume Sums}

Given a triangulation $T$ of a $d$-dimensional polytope $K$ and a time interval $[a,b]$, $a\neq b$, suppose $\Delta_T:[a,b] \times K \rightarrow K$ is a deformation
that is: 
\begin{enumerate}
\item \emph{time-linear}: $\Delta_T(t,x)$ is linear in $t$, with $\Delta_T(0,x)=x$, and
\item \emph{simplex-linear}: $\Delta_T(t,x)$ is linear in $x$ when restricted to a simplex in $T$.
\end{enumerate}
We use $\Delta_T(t,\sigma)$ to denote the image of a simplex $\sigma$ at a particular time $t$.  
For a deformation $\Delta_T$, we can define the volume sum
$$
S_{T}(t) = \sum_{\sigma \in T} \vol (\Delta_T(t,\sigma))
$$
which is the sum of the $d$-dimensional volumes of deformed simplices as a function of time. Throughout the paper the term \emph{volume} will mean \emph{signed volume} that takes orientation into account, so that at time $t=0$ all simplex volumes are non-negative.
For notational simplicity we write 
the sum over all simplices of $T$, though the only non-zero terms in the sum will come from $d$-dimensional simplices of $T$.  This volume sum can be defined similarly for any subtriangulation of $T$.

McLennan and Tourky's proof of Sperner's lemma rests on what we will call the McLennan-Tourky observation about such deformations.
\begin{thm}[McLennan-Tourky observation]
Given a triangulation $T$ of a polytope $K$ and a time-linear simplex-linear deformation $\Delta_T:[a,b]\times K \rightarrow K$, if the deformation keeps each vertex of $T$ in its minimal face in $K$, then the volume sum $S_{T}(t)$ is constant and equal to $\vol(K)$ for all $t \in [a,b]$.
\end{thm}
The observation follows immediately from Facts (1), (2), and (3) above and is proved in \cite{mclennan-tourky}.

\section{Tucker's Lemma}

As in the hypothesis of Tucker's lemma, let $P$ be a $d$-dimensional cross-polytope and $\partial P$ its boundary, endowed with: 
\begin{enumerate}
\item a triangulation $T_P$ that is symmetric on $\partial P$, and 
\item a labeling $\ell: V(T) \rightarrow \{\pm 1,...,\pm d\}$ that is anti-symmetric on $\partial P$: antipodal vertex labels sum to zero, i.e., $\ell(-v)=\ell(v)$ for each vertex $v \in \partial P \cap V(T)$.
\end{enumerate}

The $d$-dimensional cross-polytope $P$ has $2d$ extreme points that come in complementary pairs: $e_{+i}, e_{-i}$ for $i=1,...,d$.  It also has $2^d$ facets, each spanned by $d$ extreme points formed by choosing exactly one of each complementary pair.

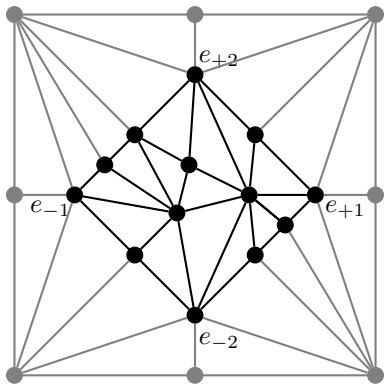
\begin{figure}[ht]
\begin{tikzpicture}[scale=0.8]
	\vertex (v1) at (2,0) {};
    \node[label=$e_{+2}$] at (2.4,-0.2) {};
    \vertex (v2) at (4,-2) {};
    \node[label=$e_{+1}$] at (4.5,-2.7) {};
    \vertex (v3) at (2,-4) {};
    \node[label=$e_{-2}$] at (2.4,-4.9) {};
    \vertex (v4) at (0,-2) {};
    \node[label=$e_{-1}$] at (-0.4,-2.7) {};    
   
    \vertex (v5) at (3,-1) {};
    \vertex (v6) at (1,-3) {};
    \vertex (v7) at (1,-1) {};
    \vertex (v8) at (3,-3) {};
    %v9 gone
    %v10 gone
    \vertex (v11) at (0.5,-1.5) {};
    \vertex (v12) at (3.5,-2.5) {};
    
    \vertex (v13) at (1.7,-2.3) {};
    \vertex (v14) at (2.9,-2) {};
    \vertex (v15) at (1.9,-1.5) {};
    
    \vertex[gray] (d1) at (-1,1) {};
    \vertex[gray] (d2) at (5,1) {};
    \vertex[gray] (d3) at (5,-5) {};
    \vertex[gray] (d4) at (-1,-5) {};
    
    \vertex[gray] (m1) at (2,1) {};
%    \node[label=$m_{+2}$] at (2,1.5) {};
%    \node[label=$+2$] at (2,1) {};
    \vertex[gray] (m2) at (5,-2) {};
%    \node[label=$-1$] at (5.5,-2.3) {};
%    \node[label=$m_{-1}$] at (5.5,-2.8) {};
    \vertex[gray] (m3) at (2,-5) {};
%    \node[label=$-2$] at (2,-6) {};
%    \node[label=$m_{-2}$] at (2,-6.5) {};
    \vertex[gray] (m4) at (-1,-2) {};
%    \node[label=$+1$] at (-1.5,-2.3) {};
%    \node[label=$m_{+1}$] at (-1.5,-2.8) {};
    
	 \tikzset{EdgeStyle/.style={-}}
	\Edge(v1)(v2)
    \Edge(v2)(v3)
    \Edge(v3)(v4)
    \Edge(v4)(v1)
    
    \Edge(v1)(v14)
    \Edge(v2)(v14)
    \Edge(v5)(v14)
    \Edge(v7)(v13)
    \Edge(v13)(v15)
    
    \Edge(v6)(v13)
    \Edge(v4)(v13)
    \Edge(v3)(v13)
    \Edge(v3)(v14)
    \Edge(v11)(v13)
    \Edge(v7)(v15)
    \Edge(v15)(v14)
    \Edge(v14)(v12)
    \Edge(v12)(v14)
    \Edge(v8)(v14)
    \Edge(v14)(v13)
    \Edge(v15)(v1)
    
\tikzset{EdgeStyle/.style={-, gray}}
    
    \Edge(d1)(d2)
    \Edge(d2)(d3)
    \Edge(d3)(d4)
    \Edge(d4)(d1)
    \Edge(d2)(v1)
    \Edge(d2)(v5)
    \Edge(d2)(v2)
    \Edge(d3)(v2)
    \Edge(d3)(v12)
    \Edge(d3)(v8)
    \Edge(d3)(v3)
    \Edge(d4)(v3)
    \Edge(d4)(v6)
    \Edge(d4)(v4)
    \Edge(d1)(v4)
    \Edge(d1)(v11)
    \Edge(d1)(v7)
    \Edge(d1)(v1)

    \Edge(m1)(v1)
    \Edge(m2)(v2)
    \Edge(m3)(v3)
    \Edge(m4)(v4);
%    \draw [decorate,decoration={brace,amplitude=10pt,raise=4pt},xshift=1pt,yshift=0pt]
%(5.5,1) -- (5.5,-5) node [black,midway,xshift=0.8cm] {
%$2$};
\end{tikzpicture}

\caption{A triangulation $T_P$ of a diamond ($2$-dimensional cross-polytope) $P$ 
%(same as in Figure \ref{figure:tucker}) 
extended to a triangulation $T$ of a square $C$.}
\label{figure:generalcase}
\end{figure}

Embed $P$ in the \emph{interior} of some polytope, which we may as well take to be the dual polytope of $P$, a cube $C$ large enough to enclose $P$.  Figure \ref{figure:generalcase} shows the situation for dimension $2$, and Figure \ref{figure:higherdimensions} shows the situation for higher dimensions.
Extend the triangulation $T_P$ of $P$ to a triangulation $T$ of $C$ in any convenient way.

\begin{figure}[ht]
\begin{tikzpicture}[thick,scale=3]

\coordinate (O) at (0,-0.5,-0.5);
\coordinate (A) at (0,0.5,-0.5);
\coordinate (B) at (0,0.5,0.5);
\coordinate (C) at (0,-0.5,0.5);
\coordinate (D) at (1,-0.5,-0.5);
\coordinate (E) at (1,0.5,-0.5);
\coordinate (F) at (1,0.5,0.5);
\coordinate (G) at (1,-0.5,0.5);

\draw[blue,fill=yellow!80] (O) -- (C) -- (G) -- (D) -- cycle;% Bottom Face
\draw[blue,fill=blue!30] (O) -- (A) -- (E) -- (D) -- cycle;% Back Face
\draw[blue,fill=red!10] (O) -- (A) -- (B) -- (C) -- cycle;% Left Face
\draw[blue,fill=red!20,opacity=0.8] (D) -- (E) -- (F) -- (G) -- cycle;% Right Face
\draw[blue,fill=red!20,opacity=0.6] (C) -- (B) -- (F) -- (G) -- cycle;% Front Face
\draw[blue,fill=red!20,opacity=0.8] (A) -- (B) -- (F) -- (E) -- cycle;% Top Face

%% Following is for debugging purposes so you can see where the points are
%% These are last so that they show up on top
%\foreach \xy in {O, A, B, C, D, E, F, G}{
%    \node at (\xy) {\xy};
%}

\coordinate (A1) at (0,0);
\coordinate (A2) at (0.6,0.2);
\coordinate (A3) at (1,0);
\coordinate (A4) at (0.4,-0.2);
\coordinate (B1) at (0.5,0.5);
\coordinate (B2) at (0.5,-0.5);
%    \node[label=$P$] at (0.3,-0.1,0) {};
%    \node[label=$C$] at (0,0.05,0) {};

\begin{scope}[thick,dashed,,opacity=0.6]
\draw (A1) -- (A2) -- (A3);
\draw (B1) -- (A2) -- (B2);
\end{scope}
\draw[fill=cof,opacity=0.6] (A1) -- (A4) -- (B1);
\draw[fill=pur,opacity=0.6] (A1) -- (A4) -- (B2);
\draw[fill=greeo,opacity=0.6] (A3) -- (A4) -- (B1);
\draw[fill=greet,opacity=0.6] (A3) -- (A4) -- (B2);
\draw (B1) -- (A1) -- (B2) -- (A3) --cycle;
\end{tikzpicture}
\caption{Embedding a cross-polytope $P$ in a cube $C$.}
\label{figure:higherdimensions}
\end{figure}
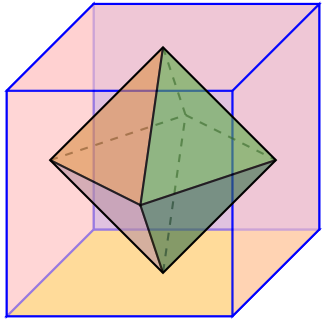

%Let $T$ be a triangulation of $C$ and let $T|_P$ denote its restriction to $P$, etc.
Consider a deformation $\Delta_T:[0,1]\times C \rightarrow C$ defined on vertices of the triangulation $T$ as follows:
$$
\Delta_T(t,v) = 
\begin{cases}
  (1-t)v + t e_{\ell(v)} & \text{ if } v \in V(T_P) \\
  v &  \text{ if $v \notin V(T_P)$ }
\end{cases}
$$
and define $\Delta_T(t,x)$ for any $x \in C$ by linear extension. Then $\Delta_T$ is both time-linear and simplex-linear, and it fixes the boundary of $C$.  At $t=0$, it is easy to see that $\Delta_T$ is the identity on all of $C$.  If there are no complementary edges in $T_P$, then as $t$ runs from $0$ to $1$, the simplices of $T_P$ collapse onto whole facets of 
$\partial P$, since vertices in $T_P$ map to the extreme points of $P$.  The simplices outside $P$ may deform but the boundary of $C$ is fixed.  So the McLennan-Tourky observation applies; the volume sum $S_{T}(t)$ remains constant over $[0,1]$ and equal to $\vol(C)$.

Now we consider a slightly different labeled triangulation of $P$, which we denote by $T^*_P$: it has exactly the same labeled triangulation structure on $\partial P$ as $T$ does, but only one vertex $e_0$ in the interior of $P$, placed at the origin and labeled $0$.  All maximal simplices of $T^*_P$ are cones from $e_0$ to the simplices in $\partial P$.  See Figure \ref{figure:cones}.

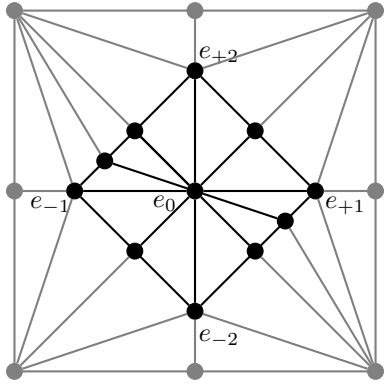
\begin{figure}[ht]
\begin{tikzpicture}[scale=0.8]
	\vertex (v1) at (2,0) {};
    \vertex (v2) at (4,-2) {};
    \vertex (v3) at (2,-4) {};
    \vertex (v4) at (0,-2) {};

    \node[label=$e_{+2}$] at (2.4,-0.2) {};
    \node[label=$e_{+1}$] at (4.5,-2.7) {};
    \node[label=$e_{-2}$] at (2.4,-4.9) {};
    \node[label=$e_{-1}$] at (-0.4,-2.7) {};    
    \node[label=$e_{0}$] at (1.5,-2.65) {};    

    \vertex (v5) at (3,-1) {};
    \vertex (v6) at (1,-3) {};
    \vertex (v7) at (1,-1) {};
    \vertex (v8) at (3,-3) {};
    %v9 gone
    %v10 gone
    \vertex (v11) at (0.5,-1.5) {};
    \vertex (v12) at (3.5,-2.5) {};
    
    \vertex (v13) at (2,-2) {};
    
    \vertex[gray] (d1) at (-1,1) {};
    \vertex[gray] (d2) at (5,1) {};
    \vertex[gray] (d3) at (5,-5) {};
    \vertex[gray] (d4) at (-1,-5) {};
    
    \vertex[gray] (m1) at (2,1) {};
%    \node[label=$m_{+2}$] at (2,1.5) {};
%    \node[label=$+2$] at (2,1) {};
    \vertex[gray] (m2) at (5,-2) {};
%    \node[label=$-1$] at (5.5,-2.3) {};
%    \node[label=$m_{-1}$] at (5.5,-2.8) {};
    \vertex[gray] (m3) at (2,-5) {};
%    \node[label=$-2$] at (2,-6) {};
%    \node[label=$m_{-2}$] at (2,-6.5) {};
    \vertex[gray] (m4) at (-1,-2) {};
%    \node[label=$+1$] at (-1.5,-2.3) {};
%    \node[label=$m_{+1}$] at (-1.5,-2.8) {};
    
	 \tikzset{EdgeStyle/.style={-}}
	\Edge(v1)(v2)
    \Edge(v2)(v3)
    \Edge(v3)(v4)
    \Edge(v4)(v1)
    
    \Edge(v1)(v13)
    \Edge(v2)(v13)
    \Edge(v5)(v13)
    \Edge(v7)(v13)
    
    \Edge(v6)(v13)
    \Edge(v4)(v13)
    \Edge(v3)(v13)
    \Edge(v11)(v13)
    \Edge(v12)(v13)
    \Edge(v7)(v13)
    \Edge(v8)(v13)
    
\tikzset{EdgeStyle/.style={-, gray}}
    
    \Edge(d1)(d2)
    \Edge(d2)(d3)
    \Edge(d3)(d4)
    \Edge(d4)(d1)
    \Edge(d2)(v1)
    \Edge(d2)(v5)
    \Edge(d2)(v2)
    \Edge(d3)(v2)
    \Edge(d3)(v12)
    \Edge(d3)(v8)
    \Edge(d3)(v3)
    \Edge(d4)(v3)
    \Edge(d4)(v6)
    \Edge(d4)(v4)
    \Edge(d1)(v4)
    \Edge(d1)(v11)
    \Edge(d1)(v7)
    \Edge(d1)(v1)

    \Edge(m1)(v1)
    \Edge(m2)(v2)
    \Edge(m3)(v3)
    \Edge(m4)(v4);
%    \draw [decorate,decoration={brace,amplitude=10pt,raise=4pt},xshift=1pt,yshift=0pt]
%(5.5,1) -- (5.5,-5) node [black,midway,xshift=0.8cm] {
%$2$};
\end{tikzpicture}
\caption{A triangulation $T^*_P$ of $P$, extended to a triangulation $T^*$ of $C$.}
\label{figure:cones}
\end{figure}

Extend $T^*_P$ to a triangulation $T^*$ of $C$ in the exact same way as we extended $T_P$ to $T$ so that on $E=C \setminus \mbox{int}(P)$, the triangulation $T^*$ is identical to $T$.  For convenience later, we denote the restrictions of $T$ and $T^*$ to $E$ by $T_E$ and $T^*_E$, respectively, and note that $T_E=T^*_E$.
The deformation $\Delta_{T^*}:[0,1]\times C \rightarrow C$ can be defined as before:
$$
\Delta_{T^*}(t,v) = 
\begin{cases}
  (1-t)v + t e_{\ell(v)} & \text{ if } v \in V(T^*_P) \\
  v &  \text{ if $v \notin V(T^*_P)$ }
\end{cases}
$$
on vertices of ${T^*}$, then we can extend the definition linearly over simplices.  Then $\Delta_{T^*}$ is both time-linear and simplex-linear and fixes the boundary of $C$.  At $t=0$, $\Delta_{T^*}$ is the identity on $C$.  And as $t$ runs from $0$ to $1$, the deformation $\Delta_{T^*}$ will keep $e_0$ fixed and move vertices in $\partial P$ to one of the extreme points of $P$.  The simplices outside $P$ deform but the boundary of $C$ is fixed.  Thus the McLennan-Tourky observation applies and the volume sum $S_{T^*}(t)$ remains constant over $[0,1]$ and equal to $\vol(C)$.

Now compare these two deformations $\Delta_T$ and $\Delta_{T^*}$.  In both cases, their volume sums over $C$ are equal, having remained constant at $\vol(C)$ as time runs from $0$ to $1$.  So $S_T(t)=S_{T^*}(t)=\vol(C)$.  Breaking up the volume sums into sums over triangulations in $E$ and $P$ we have:
$$
S_{T_P}(t) + S_{T_E}(t) = S_{T^*_P}(t) + S_{T^*_E}(t).
$$
In both cases the volume sums over $E$ may not remain constant, but because $T$ and $T^*$ are deforming in exactly the same way, i.e., $\Delta_{T} = \Delta_{T^*}$ on $E$, we see that their volume sums stay equal: $S_{T_E}(t)=S_{T^*_E}(t)$.
Hence, the volumes sums over $P$ must remain equal, and in particular, at time $t=1$:
$$
S_{T_P}(1)=S_{T^*_P}(1).
$$

The following proposition shows the conclusion of Tucker's lemma holds if we show $S_{T_P}(1) \neq 0$.

\begin{prop}
\label{prop:nonzerovolume}
If $S_{T_P}(1) \neq 0$, then $T_P$ contains a complementary edge.
\end{prop}

\begin{proof}
If the volume sum $S_{T_P}(1) = \sum_{\sigma \in P} \vol(\Delta_{T_P}(1, \sigma))$ is non-zero, then at least one of the terms is non-zero: there must exist a simplex $\sigma$ whose deformed volume at time $1$ is non-zero. This can only happen if the vertices of $\sigma$ have $d+1$ distinct labels. These are chosen from $d$ complementary pairs of labels, so $\sigma$ must have a complementary edge.
\end{proof}

So it will suffice to show $S_{T^*_P}(1) \neq 0$.  We consider $T^*_P$ instead of $T_P$ for reasons that will be clear shortly.

We will appeal to the use of the \emph{degree} of a simplicial map of spheres.  It is a standard fact that every simplicial map $f$ of a $d$-dimensional sphere with triangulation $K$ to a $d$-dimensional sphere with triangulation $L$ has a well-defined integer associated with it, called the \emph{degree} of $f$.  We review that briefly here.  
For any $d$-dimensional simplex $\sigma$ in $L$, let $p(\sigma)$ be the number of pre-images of $\sigma$ under $f$ that map to $\sigma$ with positive orientation (the outside of the sphere maps to the outside of the sphere), and let $n(\sigma)$ be the number of pre-images of $\sigma$ under $f$ that map to $\sigma$ with negative orientation (the outside of the sphere maps to the inside of the sphere).  Then the quantity $p(\sigma)-n(\sigma)$ is called the \emph{degree} of $f$ because it is independent of $\sigma$.  This can be seen by noting that if $\sigma$ and $\tau$ are adjacent $d$-simplices in $L$ sharing a common $(d-1)$-face $\gamma$, then their pre-images of $\sigma$ and $\tau$ in $K$ are matched by paths of $d$-dimensional or $(d-1)$-dimensional simplices in $K$ that are pre-images of $\gamma$.  Pre-images of $\sigma$ that are matched to each other will map to $\sigma$ with opposite orientation so they contribute nothing to the sum $p(\sigma)-n(\sigma)$.  Similarly, pre-images of $\tau$ that are matched to each other will map to $\tau$ with opposite orientation and contribute nothing to the sum $p(\tau)-n(\tau)$.  What remains are pre-images of $\sigma$ that are mapped to pre-images of $\tau$ with the same orientation.  So $p(\sigma)-n(\sigma)=p(\tau)-n(\tau)$ and the degree of $f$ is well-defined.

\begin{prop}
\label{prop:degree}
If $T_P$ has no complementary edge in $\partial P$, then $S_{T^*_P}(1) \neq 0$.
\end{prop}

\begin{proof}
Let $T^\circ_P$ denote the triangulation of $P$ formed by taking facets of $P$ and coning to $e_0$ at the origin.
If there are no complementary edges in $T_{\partial P}$, the deformation $\Delta_{T^*}$ at time $t=1$ will induce a simplicial map 
$T^*_P \rightarrow T^\circ_P$
whose restriction to $\partial P$ is a simplicial map of spheres 
$f:T^*_{\partial P} \rightarrow T^\circ_{\partial P}$.

\begin{figure}[ht]
\begin{tikzpicture}[scale=0.8]
	\vertex (v1) at (2,0) {};
    \vertex (v2) at (4,-2) {};
    \vertex (v3) at (2,-4) {};
    \vertex (v4) at (0,-2) {};

    \node[label=$e_{+2}$] at (2.4,-0.2) {};
    \node[label=$e_{+1}$] at (4.5,-2.7) {};
    \node[label=$e_{-2}$] at (2.4,-4.9) {};
    \node[label=$e_{-1}$] at (-0.4,-2.7) {};    
    \node[label=$e_{0}$] at (1.7,-2.7) {};    

    \vertex (v13) at (2,-2) {};
        
 \tikzset{EdgeStyle/.style={-}}
	\Edge(v1)(v2)
    \Edge(v2)(v3)
    \Edge(v3)(v4)
    \Edge(v4)(v1)
    
    \Edge(v4)(v13)
    \Edge(v3)(v13)
    \Edge(v2)(v13)
    \Edge(v1)(v13)
    
\end{tikzpicture}
\caption{The triangulation $T^\circ_P$ of $P$.}
\label{figure:boundaryP}
\end{figure}
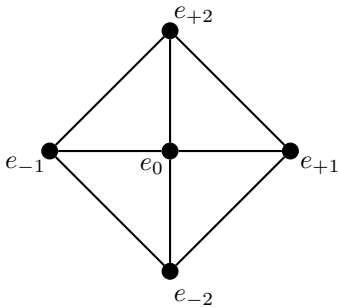

So $f$ has a well-defined degree $\deg(f)$, which is the multiplicity of the preimages of any simplex $\sigma$, counted with orientation signs.  (Our figures show examples for a $2$-dimensional cross-polytope $P$, and here the degree of $f$ on the $1$-dimensional boundary of $P$ just becomes the \emph{winding number} of $T_{\partial P}$ around $T^\circ_{\partial P}$.)

Any simplex in $T^\circ_P$ is a cone of $e_0$ over a simplex $\sigma$ in $T^\circ_{\partial P}$ and that cone appears in $S_{T^*_P}(1)$ with multiplicity $\deg(f)$.  Hence:
$$
S_{T^*_P}(1)= \deg(f) \cdot \vol(P).
$$

The Tucker labeling condition causes $f:T^*_{\partial P} \rightarrow T^\circ_{\partial P}$ to be antipode-preserving, and it is a well-established fact the degree of such maps is odd (e.g., see Lefschetz \cite{lefschetz} or Matou\v{s}ek \cite{matousek}). 
%We do not repeat that argument here.
But since an odd integer and $\vol(P)$ are both nonzero, this means that $S_{T^*_P}(1)$ is nonzero, as desired.
\end{proof}

We remark that the purpose of introducing the triangulation $T^*_P$ was to make the connection apparent between $\deg(f)$ and the volume sum $S_{T^*_P}(1)$, since the connection of $\deg(f)$ to the volume sum $S_{T_P}(1)$ is not so clear.  

Propositions \ref{prop:nonzerovolume} and \ref{prop:degree} together with the observation $S_{T_P}(1)=S_{T^*_P}(1)$
provide a proof Tucker's lemma in any dimension, as desired.

%************************************************

%\nocite{*}
\bibliographystyle{plain}
\bibliography{tucker-volume} 

\begin{thebibliography}{10}

\bibitem{baker}
James~K. Baker.
\newblock A combinatorial proof of {T}ucker's lemma for the {$n$}-cube.
\newblock {\em J. Combinatorial Theory}, 8:279--290, 1970.

\bibitem{freund-todd}
Robert~M. Freund and Michael~J. Todd.
\newblock A constructive proof of {T}ucker's combinatorial lemma.
\newblock {\em J. Combin. Theory Ser. A}, 30(3):321--325, 1981.

\bibitem{lefschetz}
Solomon Lefschetz.
\newblock {\em Introduction to {T}opology}.
\newblock Princeton Mathematical Series, vol. 11. Princeton University Press,
  Princeton, N. J., 1949.

\bibitem{matousek}
Ji\v{r}\'i Matou\v{s}ek.
\newblock {\em Using the Borsuk-Ulam Theorem}.
\newblock Springer-Verlag, 2003.

\bibitem{mclennan-tourky}
Andrew McLennan and Rabee Tourky.
\newblock Using volume to prove {S}perner's lemma.
\newblock {\em Econom. Theory}, 35(3):593--597, 2008.

\bibitem{meunier}
Fr{\'e}d{\'e}ric Meunier.
\newblock Discrete splittings of the necklace.
\newblock {\em Math. Oper. Res.}, 33(3):678--688, 2008.

\bibitem{nash}
John Nash.
\newblock Non-cooperative games.
\newblock {\em Ann. of Math. (2)}, 54:286--295, 1951.

\bibitem{prescott-su}
Timothy Prescott and Francis~Edward Su.
\newblock A constructive proof of {K}y {F}an's generalization of {T}ucker's
  lemma.
\newblock {\em J. Combin. Theory Ser. A}, 111(2):257--265, 2005.

\bibitem{simmons-su}
Forest~W. Simmons and Francis~Edward Su.
\newblock Consensus-halving via theorems of {B}orsuk-{U}lam and {T}ucker.
\newblock {\em Math. Social Sci.}, 45(1):15--25, 2003.

\bibitem{su}
Francis~Edward Su.
\newblock Rental harmony: {S}perner's lemma in fair division.
\newblock {\em Amer. Math. Monthly}, 106(10):930--942, 1999.

\bibitem{tucker}
A.~W. Tucker.
\newblock Some topological properties of disk and sphere.
\newblock In {\em Proc. {F}irst {C}anadian {M}ath. {C}ongress, {M}ontreal,
  1945}, pages 285--309. University of Toronto Press, Toronto, 1946.

\end{thebibliography}

\vfill

\end{document}